\documentclass{amsart}

\usepackage{xspace}

\usepackage{setspace}

\providecommand{\norm}[1]{\ensuremath{\left\lVert#1\right\rVert}\xspace}

\providecommand{\C}{\ensuremath{{\mathbb C}}\xspace}
\providecommand{\R}{\ensuremath{{\mathbb R}}\xspace}

\providecommand{\Z}{\ensuremath{{\mathbb Z}}\xspace}
\providecommand{\Q}{\ensuremath{{\mathbb Q}}\xspace}

\providecommand{\eqsp}{\quad}

\DeclareMathOperator{\im}{Im}

\theoremstyle{plain}
\newtheorem{theorem}{Theorem}[section]

\newtheorem{lemma}[theorem]{Lemma}
\newtheorem{cor}[theorem]{Corollary}

\theoremstyle{definition}

\newtheorem{remark}[theorem]{Remark}

\numberwithin{equation}{section}

\author{Yu. Lyubich}
\author{A. Tsukerman}
\title[Idempotents]{Idempotents in nonassociative  
algebras and eigenvectors of quadratic operators}

\address{Corresponding Author: Yu. Lyubich, Department of Mathematics, Technion, Haifa 32000, Israel}
\email{lyubich@tx.technion.ac.il}

\address{A. Tsukerman, 30 Port Royal, Foster City, CA 94404}

\date{}

\begin{document}


\bibliographystyle{amsplain}


\begin{abstract} 
Let $F$ be a field, char$(F)\neq 2$. Then every finite-dimensional $F$-algebra has 
either an idempotent or an absolute nilpotent if and 
only if over $F$ every polynomial of odd degree has a root in $F$. This is also 
necessary and sufficient for existence of eigenvectors for all
quadratic operators in finite-dimensional spaces over $F$. 
\end{abstract}

\maketitle

\section{Introduction }
\label{sec:intro}

Let $F$ be a field, and let $A$ be an $F$-{\em algebra}, i.e. a vector space 
over $F$ endowed by a multiplication that is a bilinear 
mapping $(x,y)\mapsto xy$ from the Cartesian square  $A\times A$ into $A$. 
The associativity, the commutativity and the unitality are not assumed 
in this general definition. 

Let $\dim A =n$, $ 1\leq n<\infty$, and let 
$(e_i)_{i=1}^n$ be a basis of $A$. Then the coordinates 
of the product $z=xy$ are 
\begin{equation}
\label{eq:1}
\zeta_j = \sum_{i,k=1}^n \alpha_{ik,j} \xi_i\eta_k, \eqsp 1\leq j\leq n, 
\end {equation}     
where $\xi_i,\eta_k$ are the coordinates of $x$ and $y$, respectively,
and $\alpha_{ik,j}\in F$ are the uniquely determined coefficients. 
Every cubic matrix $[\alpha_{ik,j}]$ can be realized in this 
context but the algebra is commutative if and only if the square 
matrices $[\alpha_{ik,j}]_{i,k = 1}^n$, $1\leq j\leq n$, are symmetric. 

Any algebra $A$ determines a {\em quadratic operator} $Vx=x^2$ in the
underlying vector space . This mapping $A\rightarrow A$ is homogeneous of degree 2, i.e. 
$V(\alpha x) = \alpha^2Vx$ for all $x\in A$ and all $\alpha\in F$. 
With $y=x$ in \eqref{eq:1} we have 
\begin{equation}
\label{eq:2}
\zeta_j = \sum_{i,k=1}^n \alpha_{ik,j} \xi_i\xi_k, \eqsp 1\leq j\leq n, 
\end {equation}    
where $z_j$ are the coordinates of $Vx$. If char$(F)\neq 2$ then the
coefficients $\alpha_{ik,j}$ in \eqref{eq:2} can be symmetrized, thus 
in this case all quadratic operators come from the commutative
algebras.    

Similarly to the linear case, a vector $x\neq 0$ is called an {\em eigenvector} of $V$ 
if there exists an {\em eigenvalue} $\lambda\in F$ such that $Vx = \lambda x$, i.e.  
\begin{equation}
\label{eq:3}
\sum_{i,k=1}^n \alpha_{ik,j} \xi_i\xi_k = \lambda\xi_j, \eqsp 1\leq j\leq n.
\end {equation} 
The eigenvectors of $V$ are just the basis vectors of 
the one-dimensional subalgebras of the algebra $A$. 

The set $\sigma(V)$ of eigenvalues of $V$ can be called its 
{\em spectrum}. However, only two following cases are substantial. 

1) $\lambda =1$. In this case every eigenvector
$x$ is a fixed point of $V$ , so $x^2=x$, i.e. $x$ is an
{\em idempotent} in the generating algebra $A$. 

2) $\lambda =0$. In this case every eigenvector
$x$ is a root of $V$, so $x^2=0$, i.e. $x$ is an {\em absolute nilpotent}. 

Now consider any $\lambda\in\sigma(V)$, and let  $x$ be a corrsponding
eigenvector. Then there exists $z\in$ Span$(x)\setminus\{0\}$
which is either an idempotent or an absolute nilpotent in $A$.
Hence, the intersection  $\sigma_p(V)=\sigma(V)\cap{\{0,1\}}$
is not empty if and only if $\sigma(V)\neq\emptyset$.  
Moreover, if $\sigma_p(V)= \{0,1\}$ then $\sigma(V)=F$; if $\sigma_p(V)= \{1\}$ then 
$\sigma(V)=F\setminus\{0\}$ ; if $\sigma_p(V)= \{0\}$ then $\sigma(V)=\{0\}$.    

We say that a field $F$ is of {\em class} $\Lambda$ if every finite-dimensional 
$F$-algebra contains either an idempotent or an absolute
nilpotent. Equivalently, $F\in\Lambda$ means that in 
every finite-dimensional vector space over $F$ every quadratic
operator has an eigenvector, i.e. its spectrum is not empty. 

In [Ly] it was proven that $\R\in\Lambda$. The proof is topological and very short. 
Let us reproduce it here.  

Let $\norm{.}$ be an Euclidean norm in $\R^n$, and let $S=\{x:\norm{x}=1\}$. 
One can assume that $Vx\neq 0$ for all $x\neq 0$. Then $\hat{V}x = Vx/\norm{Vx}$ 
is a smooth mapping $S\rightarrow S$. Its topological degree is even since  
$\hat{V}(-x)=\hat{V}x$. Therefore, $\hat{V}$ is not homotopic to the identity mapping 
$S\rightarrow S$. Hence, there are $x\in S$ and  $\tau\in(0,1)$ such that 
$(1-\tau)\hat{V}x+\tau x = 0$. Thus, $x$ is an eigenvector of $V$. 

Note that if {\em a field $F$ is of class $\Lambda$ 
then every its finite (i.e., finite-dimensional) extension $\Phi$ is also of 
class $\Lambda$}. Indeed, let $A$ be an $l$-dimensional
$\Phi$-algebra. Then it is an $m$-dimensional $F$-algebra  
with $m=l[\Phi:F]\equiv l\dim_F\Phi$. Every idempotent or
absolute nilpotent in $F$-setting remains as is when $F$ extends to $\Phi$.

In particular, {\em $\C\in\Lambda$ since $\R\in\Lambda$}. 

In the present paper we establish a necessary and a sufficient
condition for $F\in\Lambda$. These conditions coincide if 
char$(F)\neq 2$.

\section{Results}
\label{sec:res}

Denote by $\Gamma$ the class of fields 
$F$ such that every polynomial of odd degree with coefficients from $F$ has a root in $F$.
For $F=\R$ this property is the only topological ingredient of the Gauss proof (1815) 
of the ``Fundamental Theorem of Algebra''. 
\begin{theorem}
\label{thm:acf}
Every algebraically closed field $F$ is of class $\Lambda$. 
\end{theorem}
Our Main Theorem is 
\begin{theorem}
\label{thm:1}
$\Gamma'\subset\Lambda\subset\Gamma$ where $\Gamma ' = \Gamma\cap\{F: {\rm char}(F)\neq 2 \}$. 
\end{theorem}
\begin{cor}
\label{cor:2}
If {\rm char}$(F)\neq 2$ then $F\in\Lambda$ if and only if $F\in\Gamma$.
\end{cor}
With the additional condition char$(F)\neq 2$ Theorem \ref{thm:acf} follows
from Corollary \ref{cor:2}. Also, the latter yields  
a ''real '' counterpart of Theorem \ref{thm:acf}.
\begin{cor}
\label{cor:3}
Every formally real and really closed field $F$ is of class $\Lambda$. 
\end{cor}
Indeed, in this case $F\in\Gamma'$ since char$(F) = 0$ and $F\in\Gamma$, see e.g. [La], Ch.11. 

In several corollaries below we use the inclusion $\Lambda\subset\Gamma$. 
For instance, {\em the field $\Q$ of rational numbers is not of class $\Lambda$} 
since the polynomial $\alpha^3-2$ has no roots in $\Q$. More generally, we have 
\begin{cor}
\label{prop:3}
Every finite extension $F\supset\Q$ is not of class $\Lambda$. 
\end{cor}
\begin{proof}
Assume $F\in\Lambda$. Let $\Phi$ be a finite extension of $F$ normal over $\Q$. 
With $n = [\Phi:\Q]$ let $m$ be any odd number, $m>n$. 
By Eisenstein's test there exists an irreducible 
polynomial $f$ of degree $m$ over $\Q$. This $f$ has no roots in
$\Phi$, otherwise, all $m$ roots of $f$ belong to $\Phi$,  
so $m\leq n$. Since $f$ is a polynomial over $\Phi$, we see that 
$\Phi\notin\Gamma$. Hence, $\Phi\notin\Lambda$, a fortiori,
$F\notin\Lambda$, a contradiction. 
\end{proof}
The same argument yields  
\begin{cor}
\label{prop:4}
Let $p$ be a prime number, and let $\Q_p$ be the field of $p$-adic numbers. Every  
finite extension $F\supset\Q_p$ is not of class $\Lambda$. 
\end{cor}
\begin{cor}
\label{cor:a}
Every finite field $F$ is not of class $\Lambda$. 
\end{cor}
\begin{proof}
If $q=$card$(F)$ then $\alpha^q-\alpha + 1=1\neq 0$ for all $\alpha\in F$. This 
polynomial is of odd degree if char$F\neq 2$. If char$F=2$ then such a
polynomial is $\alpha^{q+1}-\alpha^2+1$.
\end{proof}
\begin{cor}
\label{cor:aa}
For every field $K$ and every $n\geq 1$ the field
$F=K(t_1,\cdots,t_n)$ of rational functions of $n$ variables 
over $K$ is not of class $\Lambda$. 
\end{cor}
\begin{proof}
$\alpha^3-t_1\neq 0$ for all $\alpha\in F$.
\end{proof}
The class $\Gamma$ is closely related to the class $\Delta$ of fields 
such that the degrees of all their finite extensions are powers of 2. 
\begin{theorem}
\label{thm:2}
$\Gamma'\subset\Delta\subset\Gamma$. 
\end{theorem}
\begin{cor}
\label{cor:20}
If char$(F)\neq 2$ then $F\in\Gamma$ if and only if $F\in\Delta$.
\end{cor}
Thus, if char$(F)\neq 2$ then each of conditions $F\in\Gamma$ and
$F\in\Delta$ is necessary and sufficient for $F\in\Lambda$. 

\section{Proofs}
\label{sec:pro}

We start with the following 
\begin{lemma}
\label{lem:0}
A field $F$ is of class $\Gamma$ if and only if the degrees of all its nontrivial finite
extensions are even.
\end{lemma}
\begin{proof}
Let $F\notin\Gamma$, so there exists a polynomial $f$ over $F$ of odd degree 
without roots in $F$. Then $f$ has an irreducible 
factor with the same properties. The corresponding extension of $F$ 
is of odd degree $d>1$, a contradiction. 

Now let $F\in\Gamma$, and let $\Phi\supset F$ be a finite extension of 
odd degree $d>1$. For $x\in\Phi\setminus F$ the degree $d_x$ of the
extension $F[x]$ is odd as a divisor of $d$. However, $d_x$ 
is the degree of the irreducible polynomial $f$ over $F$ such that $f(x)=0$. 
This is a contradiction since $f$ has a root in $F$.
\end{proof}
This Lemma immediately implies 
$\Delta\subset\Gamma$ that is a part of Theorem \ref{thm:2}. 
\begin{proof}[Proof of $\Lambda\subset\Gamma$.]
Let $F\notin\Gamma$. By Lemma \ref{lem:0} there exists an
extension $\Phi\supset F$ of odd degree $d>1$. 
Consider the projection $\pi:\Phi\rightarrow\Phi$ onto the   
one-dimensional subspace $F$. We introduce in $\Phi$ the new multiplication 
\begin{equation}
\label{eq:new}
x\circ y = (x-\pi x)(y-\pi y). 
\end{equation} 
Then $\Phi$ becomes a commutative $F$-algebra. The subspace $F$ is an ideal 
in $\Phi$ since $x\circ y =0$ for $x\in F$ and all
$y\in\Phi$. In the quotient algebra $F$-algebra $A=\Phi/F$ we consider 
the equation 
\begin{equation}
\label{eq:nev}
X\circ X = \lambda X
\end{equation} 
with $\lambda\in F$. To conclude that $F\notin\Lambda$ it suffices
to show that \eqref{eq:nev} has no solutions $X\neq 0$. Suppose to the
contrary. Then there exists $x\in \Phi\setminus F$ such that  
\[  
(x-\pi x)^2 -\lambda x -\mu =0
\]
with a $\mu\in F$. This is a quadratic equation over $F$ since $\pi x\in F$. 
Its second root is $x'=2\pi x + \lambda -x\equiv -x$ (mod$F$), so $x'\notin F$.
Hence, the extension $F[x]\supset F$ is of degree 2. On the 
other hand, $F[x]\subset\Phi$. Therefore, $d=2[\Phi:F[x]]$. Thus, $d$ is 
even, a contradiction.
\end{proof}

Now note that {\em every field $F\in\Gamma '$ is perfect}. Indeed, if char$(F)$    
is an odd prime $p$ then $F^p =F$ because of $F\in\Gamma$.   
Recall that if a field $F$ is perfect then all its normal algebraic extensions
are separable, so they are the Galois extensions that allows
us to refer to the Galois theory, see e.g. [La], Ch. 8.
\begin{proof}[Proof of the inclusion $\Gamma'\subset\Delta$]
Let $\Phi\supset F$ be a finite extension of $F$, and let
$\Psi\supset\Phi$ be a Galois extension of $F$.  If $[\Phi:F]=2^i(2k+1)$ 
then $[\Psi :F]=2^j(2l+1)$ where $j\geq i$ and $2l+1$ 
is divisible by $2k+1$. 
Consider the Galois group $G=$Gal$(\Psi/F)$
and its Sylow 2-subgroup $H$. Let $\Omega$ be the subfield of $\Psi$
consisting of the fixed points of $H$. The index of $H$ in $G$ is
$2l+1$, hence $[\Omega :F]=2l+1$. Since $F\in\Gamma$, we 
conclude that $l=0$ by Lemma \ref{lem:0}. A fortiori, $k=0$, hence $[\Phi:F]= 2^i$.  
\end{proof}
{\em Theorem \ref{thm:2} is already proven}. 
As to Theorem \ref{thm:1}, we have to prove the inclusion
$\Gamma'\subset\Lambda$. To this end (and also to prove Theorem
\ref{thm:acf}) we return to \eqref{eq:3} and consider it as a system 
of $n$ homogeneous equations of second degree   
\begin{equation}
\label{eq:4.1} 
g_j(x,\lambda)\equiv\sum_{i,k=1}^n\alpha_{ik,j}\xi_i\xi_k-\lambda\xi_j =0,
\eqsp 1\leq j\leq n, \eqsp x=(\xi_1,\cdots,\xi_n)\in F^n,  
\end{equation}
with $n+1$ unknowns $\xi_1,\cdots,\xi_n,\lambda$. With $\lambda\in F$ this
is a system of equations in the projective space $FP^n$. The solution
$(\bf{0}$:$1)\equiv(0:\cdots:0:1)$ 
is called {\em trivial }. In all nontrivial solutions $(x,\lambda)\in
FP^n$ the scalar component $\lambda$ is an eigenvalue of the quadratic
operator under consideration. Thus, our aim is to prove the existence
of a nontrivial solution in $FP^n$. 

Let us extend $FP^n$ to the projective space $\overline{F}P^{n}$ 
over the algebraic closure $\overline{F}$.
We call the system \eqref{eq:4.1} {\em generic over $F$} if the set of its
solutions in $\overline{F}P^{n}$ is finite. Then the 
number of solutions (i.e., the sum of its multiplicities) is equal to $2^n$ 
by the Bezout Theorem. A solution is called {\em simple} if its
multiplicity is equal to 1. 
\begin{lemma}
\label{lem:1} 
If the system \eqref{eq:4.1} is generic over $F$ then the trivial
solution is simple.
\end{lemma}
\begin{proof}
The affine space $\overline{F}^n$ can be identified with  
the projective Zarisky neighborhood of $(\bf{0}$:$1)$ in $\overline{F}P^{n}$ consisiting of the points $(x:1)$,
$x\in\overline{F}^n$. The corresponding restrictions of the quadratic
forms $g_j(x,\lambda)$ are the polynomials $f_j(x)=g_j(x,1)$.  

Let $R$ be the local ring of the point $0\in\overline{F}^n$, and let $M$ be its 
maximal ideal. In the quotient $\overline{F}$- vector space $M/M^2$ the images 
of $(-\xi_j)\in M$, $1\leq j\leq n$, 
constitute a basis. The images of $f_j$ are the same 
since $f_j+\xi_j\in M^2$. By Nakayama's Lemma the system $(f_j)_{j=1}^n$ 
generates the ideal $M$. Hence, the intersection index of the corresponding divisors is 1.  
\end{proof}
Let us emphasize that Lemma \ref{lem:1} is true irrespective to char$(F)$.
\begin{proof}[Proof of Theorem \ref{thm:acf}]. 
If $F=\overline{F}$ then either the set of solutions in $FP^n$ is infinite
or the number of these solutions is $2^n>1$. By Lemma \ref{lem:1} there
exists a nontrivial solution of the system \eqref{eq:4.1} in $FP^n$. 
\end{proof}
{\em Later on} $F\in\Gamma'$. Let the system \eqref{eq:4.1} be 
generic over $F$. For every solution we fix the projective coordinates 
and take them from $F$ whensoever the solution belongs to $FP^n$. Now we     
consider the Galois extension $\Psi\supset F$ containing all these 
quantities. By Theorem \ref{thm:2} we have $[\Psi:F]= 2^i$ with an $i\geq 1$. Accordingly, the group
$G$=Gal$(\Psi/F)$ is  of order $2^i$.  
Since the coefficients of all $g_j(x,\lambda)$ belong to $F$, the
group $G$ naturally acts on the set of solutions preserving their multiplicities.  
A solution is $G$-invariant if and only if it is from $FP^n$.

Let $z=(x,\lambda)\neq 0$ be a solution of a multiplicity $m$, and let
$Z$ be its $G$-orbit. All points from $Z$ are solutions of the same 
multiplicity $m$. The contribution of $Z$ into the number
of solutions equals $\hat{m}=$card${(Z)}m$. In this product the second factor
coincides with the index of the stabilizer of $z$. 
This index is $2^j$ with some $j\geq 1$ if $z$ is not a fixed point, i.e. if
this solution is not from $FP^n$. Hence, the
number of such solutions is even. Since the number of all
solutions is $2^n$, the 
number of those solutions which are from $FP^n$ is even.
By Lemma \ref{lem:1} there exists a nontrivial solution in $FP^n$.  
{\em The generic case in Theorem \ref{thm:1} is settled}. 

To complete the proof of Theorem \ref{thm:1} 
we show that an ``infinitesimal'' perturbation of the system
\eqref{eq:4.1} is generic over an extension of $F$. 
\begin{lemma} 
\label{lem:3}
Let a field $K\supset F$ is endowed by a non-Archimedean valuation, 
let $A$ be the valuation ring, and let $A^0$ be its maximal
ideal. Then for $1\leq j\leq n$ there are some linear forms $\varphi_j(x)$ with
coefficients from $A$ and some
$\varepsilon_j\in A^0$ such that the system 
\begin{equation}
\label{eq:4.2}
g_j(x,\lambda)-\varepsilon_j(\varphi_j(x))^2 = 0, \eqsp 1\leq j\leq n,
\end{equation}
is generic over $K$. 
\end{lemma}
\begin{proof}
Inductively on $m$, $1\leq m\leq n$, we construct the variety 
\[
V_m=\{(x,\lambda)\in\overline{K}P^n:
g_j(x,\lambda)-\varepsilon_j(\varphi_j(x))^2 = 0, 1\leq j\leq m\} 
\]
of dimension $\leq n-m$. Then $\dim V_n=0$, thus 
the system \eqref{eq:4.2} is generic over $K$. 

Since $g_1\neq 0$, we can take 
$\varepsilon_1=0$ to get $\dim V_1=n-1$ with any $\varphi_1$. 
Now let $1\leq m<n$, and let $\dim V_m\leq n-m$.
Consider the decomposition $V_m=\cup_{1\leq i\leq r}X_i$ into irreducible
components, and choose any point $(x_i,\lambda_i)\in X_i$, $1\leq i\leq r$.
Let $g_{m+1}(x_i,\lambda_i) = 0$ for $1\leq i\leq s$, while 
$g_{m+1}(x_i,\lambda_i)\neq 0$ for $s+1\leq i\leq r$. Since the field $K$
is infinite, there exists a linear form $\varphi_{m+1}(x)$ on $\overline{K}^n$ 
with coefficients from $A$ such that $\varphi_{m+1}(x_i)\neq 0$, $1\leq i\leq r$. 
With $\varepsilon_{m+1}\in A^0$ different from $0$ and from all fractions 
\[
\frac{g_{m+1}(x_i,\lambda_i)}{(\varphi_{m+1}(x_i))^2},\eqsp s+1\leq i\leq r,
\]
the form $g_{m+1}(x,\lambda)-\varepsilon_{m+1}(\varphi_{m+1}(x))^2$
does not vanish on each component of $V_m$. Therefore, $\dim V_{m+1}<\dim V_m$, hence $\dim V_{m+1}\leq n-m-1$.
\end{proof}
For our purposes we need to get $K\in\Gamma'$ in Lemma \ref{lem:3}.
The first step in this direction is 
the extension of $F$ to the field $L=F((t))$ whose
nonzero elements are the formal Laurent series 
\[
a(t)=t^{\nu} \sum_{k=0}^{\infty}\alpha_kt^k
\]
where $\alpha_k\in F$, $\alpha_0\neq 0$, $\nu=\nu(a)\in\Z$. On $L$ we have
the standard non-Archimedean valuation $v_0(a)= \exp(-\nu(a))$,
$a\neq 0$. The ground field $F$ is embedded in $L$ as the field of 
constants, $v_0(\alpha)= 1$ for $\alpha\in F\setminus\{0\}$.  

With the distance $v_0(a-b)$ the set $L$ is a complete metric space. 
The closed unit ball $A_L=\{a\in L:v_0(a)\leq 1\}$, i.e. the set of
regular series, is just the valuation ring of the field $L$. Its unique
maximal ideal is the open ball $A_{L}^0=\{a:v_0(a)< 1\}$. 
The residue field $R_L=A_{L}/A_{L}^0$ is isomorphic to $F$. 

The valuation $v_0$ can be extended to a non-Archimedean valuation $v$ of the
algebraic closure $\overline{L}$. The corresponding ring is 
$A_{\overline{L}}=\{x\in \overline{L}:v(x)\leq 1\}$, its   
maximal ideal is $A_{\overline{L}}^0=\{x\in \overline{L}:v(x)< 1\}$.   
Denote by $\rho$ the natural epimorphism from 
$A_{\overline{L}}$ to the residue field $R_{\overline{L}}=A_{\overline{L}}/A_{\overline{L}}^0$.

Now let $\bf{F}$ be the family of  
fields $E$ such that $L\subset E\subset\overline{L}$, and let
$\bf{F}$ be partially ordered by inclusion. On every $E\in\bf{F}$ we have 
the valuation $v|E$. The corresponding residue field is 
isomorphic to $R_{E}=\im(\rho|A_{E})$ where $A_{E}= A_{\overline{L}}\cap E$ 
is the valuation ring of the field $E$.   

Obviosly, if $E_1, E_2\in\bf{F}$ and $E_1\subset E_2$ then
$A_{E_1}\subset A_{E_2}$ and $R_{E_1}\subset R_{E_2}$.
In particular, $R_{L}\subset R_{E}\subset R_{\overline{L}}$ for  all $E\in\bf{F}$. 
By Zorn's Lemma there exists a field $K\in\bf{F}$ 
which is maximal among $E\in\bf{F}$ with $R_{E} = R_{L}$. 
Indeed, let $\{E_i\}$ be a linearly ordered subfamily of
$\bf{F}$, and let all $ R _{E_i}=R_L$. Then $\cup E_i$ is a field $E\in\bf{F}$, and 
$A_E= A_{\overline{L}}\cap(\cup E_i) = \cup(A_{\overline{L}}\cap E_i)=
\cup A_{E_i}$, hence $R_E=\cup R_{E_i}=R_L$. 
\begin{lemma} 
\label{lem:2} $K\in\Gamma'$.   
\end{lemma}
\begin{proof}  
Since char$(K)=$char$(F)\neq 2$, we actually have to prove that
$K\in\Gamma$. 
Suppose to the contrary. Then by Lemma \ref{lem:0} there exists 
a finite extension $E\supset K$
of odd degree $d>1$. One can take $E\subset\overline{L}$ since
$K\subset \overline{L}$. We have $E\in\bf{F}$, but $R_E\neq R_L$ 
by maximality of $K$ in $\bf{F}$. Hence, there exists $x\in E\setminus K$ such that 
$\rho x\notin R_L$. 
Let $f$ be an irreducible polynomial over $K$ such that $f(x)=0$.
Its  degree $d_f$ is odd since $d_f$ divides $d$.

Let us extend the initial field $L$ to a field $L_f\subset K$ by joining of all 
coefficients of $f$. Obviosly, $L_f\in\bf{F}$. The valuation
$v|L_f$ is discrete and $L_f$ is complete. Furthermore, 
$R_{L_f} = R_L$ since  $R_L\subset R_{L_f}\subset R_K=R_L$. Hence, 
$R_{L_f}$ is isomorphic to $F$, thus it is a perfect field of class $\Gamma$.  

Now we consider the field $L_f[x]\in\bf{F}$. The degree 
$[L_f[x]:L_f]$ coincides with $d_f$ since the polynomial 
$f$ is determined and irreducible over $L_f$. Hence, this degree is odd.
The residue field of $L_f[x]$ is a finite extension of $R_{L_f}$. Its degree $\delta$ divides 
$[L_f[x]:L_f]$ by Proposition 18 from [La], Ch.12. Hence, $\delta$
is odd. On the other hand, $\delta >1$ by choice of $x$. By Lemma
\ref{lem:0} $\delta$ is even, a contradiction.
\end{proof}
We also need the following lemma.
\begin{lemma} 
\label{lem:'2}
The ring $A_K$ is the direct sum of the subrings $F$ and $A_K^0$.     
\end{lemma}
\begin{proof} 
First, note that $F\subset A_K$ since $v(x)=1$ for $x\in F$. 
For the same reason  $F\cap A_K^0 = 0$. Now let $x\in A_K$. Then 
$\rho x\in R_K = R_L$. Hence, $\rho x= \rho a$ where $a\in L$, whence $x-a\in \ker(\rho|A_K)= A_K^0$. In turn,
$a=\alpha +\omega$ where $\alpha\in F$ and $\omega\in A_L^0\subset A_K^0$.    
As a result, $x=\alpha + (x-a)+\omega\in F+A_K^0$.   
\end{proof}
Since Theorem \ref{thm:2} is true in the generic case,  the system \eqref{eq:4.2}
has a nontrivial solution $(x,\lambda)\in KP^n$ by 
Lemmas \ref{lem:3} and \ref{lem:2}.
Let $x=(\xi_i)_{i=1}^n\in K^n\setminus\{0\}$, and let
$\norm{x}=\max_iv(\xi_i)$. Show that $v(\lambda)\leq\norm{x}$. Indeed,
in the opposite case the division on $\lambda^2$ in \eqref{eq:4.2} yields
\begin{equation}
\label{eq:4.3}
\zeta_j=h_j(z), \eqsp 1\leq j\leq n,
\end{equation}
where $h_j$ are some quadratic forms with coefficients from $A_{K}$
and $z=(\zeta_j)_{i=1}^n = (\xi_j/\lambda)_{i=1}^n\in A_K^0\setminus\{0\}$. 
However, from \eqref{eq:4.3} it follows that $\norm{z}\leq\norm{z}^2$,
hence $\norm{z}\geq 1$, a contradiction. 

Thus, $\max(\norm{x},v(\lambda))=\norm{x}$. Let $\norm{x}=v(\xi_1)$ for definiteness. 
By division on $\xi_1$ we get a solution $(x_1,\lambda_1)\in KP^n$
with $\norm{x_1}=1$ and $v(\lambda_1)\leq 1$. By Lemma \ref{lem:'2}
$(x_1,\lambda_1)= (\hat{x},\hat{\lambda})+(y,\omega)$ where the first
summand belongs to $F^{n+1}$ and the second one belongs to $(A_K^0)^{n+1}$.
In addition, $\hat{x}\neq 0$, otherwise $x_1=y$, so $\norm{x_1}<1$.
Since $A_K^0$ is an ideal in $A_K$, the system \eqref{eq:4.2} yields 
$g_j(\hat{x},\hat{\lambda})\in F\cap A_K^0$, $1\leq j\leq n$. This implies 
$g_j(\hat{x},\hat{\lambda})=0$, $1\leq j\leq n$, by Lemma \ref{lem:'2}
again. 
{\em Theorem \ref{thm:2} is proven completely}.

\section{Some remarks}
\label{sec:rem}

\begin{remark}
\label{rem:1}
Over any field any power-associative  
finite-dimensional algebra has an idempotent or a 
nilpotent of some degree, see [S], Proposition 3.3. In the latter case if
$x^r=0$, $r\geq 2$, but $x^{r-1}\neq 0$, then $x^s$ is an absolute
nilpotent for $s = r - [r/2]$. Thus, if we restrict the definition of
class $\Lambda$ to the power-associative
algebras then $\Lambda$ extends to the class of all fields, so our problem disappears.
The same happens trivially underthe the restriction to the unital algebras. 
\end{remark}
\begin{remark}
\label{rem:2}
The infinite-dimensional version of $\Lambda$ is
empty.  Indeed, over any field $F$ there are no idempotents neither absolute nilpotents
in the algebra of polynomials $f(t)$ such that $f(0) = 0$.
\end{remark}
\begin{remark}
\label{rem:3}
The algebra $A$ from the proof of the inclusion 
$\Lambda\subset\Gamma$ is commutative. Therefore, this inclusion 
remains valid if in the definition of $\Lambda$ we consider 
the commutative algebras only.
\end{remark}
\begin{remark}
\label{rem:4}
In the case char$(F)=0$ the Corollary \ref{thm:acf} follows from $\C\in\Lambda$
by the metamathematical Lefshets Principle.
\end{remark}


\section*{References}
\label{sec:references}

[La] S.Lang, Algebra. Addison-Wessley Publ.Co., 1965.

[Ly] Yu.I.Lyubich, The nilpotency of the Engel commutative algebras 
of dimension $n\leq 4$ over $\R$. Uspechi Mat.Nauk, {\bf 32} , no.1
(1977), 195-196 (In Russian). 

[S] R.D.Schafer, An Introduction to Nonassociative
Algebras. Acad. Press, New York -London, 1966.

\end{document}